\documentclass[12pt]{amsart}

\usepackage{amstext}
\usepackage{amsthm}
\usepackage{amsmath}
\usepackage{amssymb}
\usepackage{latexsym}
\usepackage{amsfonts}
\usepackage{graphicx}
\usepackage{texdraw}
\usepackage{graphpap}
\usepackage{enumitem}
\usepackage[pagebackref,hypertexnames=false, colorlinks, citecolor=red, linkcolor=red]{hyperref} 
\usepackage[backrefs]{amsrefs}

\input txdtools

\DeclareFontFamily{U}{wncy}{}
    \DeclareFontShape{U}{wncy}{m}{n}{<->wncyr10}{}
    \DeclareSymbolFont{mcy}{U}{wncy}{m}{n}
    \DeclareMathSymbol{\Sh}{\mathord}{mcy}{"58}
\begin{document}

\newcommand{\ci}[1]{_{ {}_{\scriptstyle #1}}}

\newcommand{\norm}[1]{\ensuremath{\|#1\|}}
\newcommand{\abs}[1]{\ensuremath{\vert#1\vert}}
\newcommand{\p}{\ensuremath{\partial}}
\newcommand{\pr}{\mathcal{P}}

\newcommand{\pbar}{\ensuremath{\bar{\partial}}}
\newcommand{\db}{\overline\partial}
\newcommand{\D}{\mathcal{D}}
\newcommand{\B}{\mathbb{B}}
\newcommand{\Sp}{\mathbb{S}}
\newcommand{\T}{\mathbb{T}}
\newcommand{\R}{\mathbb{R}}
\newcommand{\Z}{\mathbb{Z}}
\newcommand{\C}{\mathbb{C}}
\newcommand{\N}{\mathbb{N}}
\newcommand{\scrH}{\mathcal{H}}
\newcommand{\scrL}{\mathcal{L}}
\newcommand{\td}{\widetilde\Delta}

\newcommand{\La}{\langle }
\newcommand{\Ra}{\rangle }
\newcommand{\rk}{\operatorname{rk}}
\newcommand{\card}{\operatorname{card}}
\newcommand{\ran}{\operatorname{Ran}}
\newcommand{\osc}{\operatorname{OSC}}
\newcommand{\im}{\operatorname{Im}}
\newcommand{\re}{\operatorname{Re}}
\newcommand{\tr}{\operatorname{tr}}
\newcommand{\vf}{\varphi}
\newcommand{\f}[2]{\ensuremath{\frac{#1}{#2}}}


\newcommand{\entrylabel}[1]{\mbox{#1}\hfill}

\newenvironment{entry}
{\begin{list}{X}%
  {\renewcommand{\makelabel}{\entrylabel}%
      \setlength{\labelwidth}{55pt}%
      \setlength{\leftmargin}{\labelwidth}
      \addtolength{\leftmargin}{\labelsep}%
   }%
}%
{\end{list}}


\numberwithin{equation}{section}
\newtheorem{dfn}{Definition}[section]
\newtheorem{thm}{Theorem}[section]
\newtheorem{lm}[thm]{Lemma}
\newtheorem{cor}[thm]{Corollary}
\newtheorem{conj}[thm]{Conjecture}
\newtheorem{prob}[thm]{Problem}
\newtheorem{prop}[thm]{Proposition}
\newtheorem*{prop*}{Proposition}

\theoremstyle{remark}
\newtheorem{rem}[thm]{Remark}
\newtheorem*{rem*}{Remark}

\newtheorem{quest}[thm]{Question}

\title{Multilinear Dyadic Operators And Their Commutators In The Weighted Setting}

\author{Ishwari Kunwar}

\address{Ishwari J. Kunwar, School of Mathematics\\ Georgia Institute of Technology\\ 686 Cherry Street\\ Atlanta, GA USA 30332-0160}
\email{ikunwar3@math.gatech.edu}

\subjclass[2000]{Primary }

\begin{abstract} 
In this article, we investigate the boundedness properties of the multilinear dyadic paraproduct operators in the weighted setting. We also obtain weighted estimates for the multilinear Haar multipliers and their commutators with dyadic BMO functions. 
\end{abstract}
\maketitle
\setcounter{tocdepth}{1}
\tableofcontents
\section {Introduction and statement of main results}

\noindent
The purpose of this article is to investigate the boundedness properties of the multilinear dyadic operators (paraproducts and Haar multipliers) introduced in \cite{IJK} and thier commutators in the weighted setting as adopted in \cite{LOPTT}. Mainly, we use the unweighted theory of multilinear dyadic operators from \cite{IJK}, explore some useful properties of those operators, and run the machinery used in \cite{LOPTT} to obtain the corresponding weighted theory of the multilinear dyadic operators. \\

\noindent
In \cite{IJK}, the paraproduct decomposition of the pointwise product of two functions was generalized to the product of $m \geq 2$ functions that served as the motivation for defining the following multilinear dyadic operators.
\begin{itemize}
	
\item  $\displaystyle P^{\vec{\alpha}}(f_1,f_2,\ldots,f_m) = \sum_{I\in\mathcal{D}} \left(\prod_{j=1}^m f_j(I,\alpha_j)\right) h_I^{\sigma(\vec{\alpha})}, \quad \vec{\alpha} \in \{0,1\}^m \backslash\{(1,1,\ldots,1)\}.$\\

\item $\displaystyle \pi_b^{\vec{\alpha}}(f_1, f_2, \ldots, f_m) = \sum_{I \in \mathcal{D}} \La b , h_I \Ra \left(\prod_{j=1}^m f_j(I,\alpha_j)\right)  h_I^{1+\sigma(\vec{\alpha})},\quad \vec{\alpha} \in \{0,1\}^m,\, b\in BMO^d.$\\

\item  $\displaystyle T_\epsilon^{\vec{\alpha}} (f_1,f_2, \ldots,f_m) := \sum_{I\in \D} \epsilon_I \left(\prod_{j=1}^m f_j(I,\alpha_j)\right) h_I^{\sigma(\vec{\alpha})},$ \\

 $ \quad \vec{\alpha} \in =\{0,1\}^m \backslash \{(1,1,\ldots,1)\}, \, \epsilon = \{\epsilon_I\}_{I\in \D} \text{ bounded}.$\\

\item $\displaystyle [b,T_\epsilon^{\vec{\alpha}}]_i(f_1,f_2,\ldots,f_m)(x) := b(x)T_\epsilon^{\vec{\alpha}}(f_1,f_2,\ldots,f_m)(x) - T_\epsilon^{\vec{\alpha}}(f_1, \ldots, bf_i,\ldots,f_m)(x),$\\

 $ 1 \leq i \leq m$, $ \vec{\alpha} \in \{0,1\}^m \backslash\{(1,1,\ldots,1)\},\, \epsilon = \{\epsilon_I\}_{I\in \D} \text{ bounded and } b\in BMO^d.$\\
\end{itemize}

\noindent In the above definitions, $\D := \{[m2^{-k}, (m+1)2^{-k}): m,k\in \mathbb{Z}\}$ is the standard dyadic grid on $\R$ and $h_I$'s are the Haar functions defined by $h_I = \displaystyle \frac{1}{\abs{I}^{1/2}}\left(\mathsf{1}_{I_+} - \mathsf{1}_{I_-}\right),$ where $I_-$ and $I_+$ are the left and right halves of $I.$ With $\left< \;,\;\right>$ denoting the standard inner product in $L^2(\R),$ $f_i(I,0) := \left< f_i,h_I\right>$ and $\displaystyle f_i(I,1) := \La f_i, h_I^2\Ra = \frac{1}{\abs{I}} \int_I f_i,$ the average of $f_i$ over $I.$ The Haar coefficient $\La f_i, h_I\Ra$ is sometimes denoted by $\widehat{f_i}(I)$ and the average of $f_i$ over $I$ by $\La f_i \Ra_I$. For $\vec{\alpha} \in \{0,1\}^m,$ $\sigma(\vec{\alpha})$ to denotes the number of 0 components in $\vec{\alpha}$. For convenience, we will denote the set $\{0,1\}^m \backslash\{(1,1,\ldots,1)\}$ by $U_m.$\\

\noindent
The following boundedness properties of the multilinear dyadic operators were proved in \cite{IJK}:\\

\begin{itemize}
	\item 
Let $ \vec{\alpha} = (\alpha_1,\alpha_2,\ldots,\alpha_m) \in \{0,1\}^m$ and $ 1 < p_1, p_2, \ldots, p_m < \infty$ with $\sum_{j=1}^m \frac{1}{p_j} = \frac{1}{p}.$ Then
\begin{enumerate}[label = $(\alph*)$]
\item For $\vec{\alpha}  \neq (1,1,\ldots,1),$  $$ \left\Vert P^{\vec{\alpha}}(f_1,f_2,\ldots,f_m)\right\Vert_p \lesssim \prod_{j=1}^m\norm{f_j}_{p_j}.$$
\item For $\sigma(\vec{\alpha}) \leq 1,$ $$\left\Vert\pi_b^{\vec{\alpha}}(f_1,f_2,\ldots,f_m)\right\Vert_p \lesssim \norm{b}_{BMO^d}\prod_{j=1}^m\norm{f_j}_{p_j},$$ if and only if $b \in BMO^d.$\\

\noindent
For $\sigma(\vec{\alpha}) > 1,$ $$\left\Vert\pi_b^{\vec{\alpha}}(f_1,f_2,\ldots,f_m)\right\Vert_p \leq C_b \prod_{j=1}^m\norm{f_j}_{p_j},$$ if and only if $\displaystyle\sup_{I\in \D} \frac{\abs{\La b,h_I\Ra}}{\sqrt{\abs{I}}} < \infty.$
\item For $\vec{\alpha}  \neq (1,1,\ldots,1),$ $$\left\Vert T_\epsilon^{\vec{\alpha}}(f_1,f_2,\ldots,f_m)\right\Vert_p \lesssim \prod_{j=1}^m\norm{f_j}_{p_j},$$ if and only if $\norm{\epsilon}_\infty:= \displaystyle \sup_{I \in \D}\abs{\epsilon_I} < \infty.$\\
\end{enumerate}

\noindent
In each of the above cases, the operators have the corresponding weak-type boundedness from $L^{p_1} \times \cdots \times L^{p_m}\rightarrow L^{p,\infty}$ if $1\leq p_1, p_2, \ldots, p_m < \infty$.\\

\item Let $\vec{\alpha} = (\alpha_1,\alpha_2,\ldots,\alpha_m) \in U_m.$  If $b \in BMO^d \cap L^r$ for some $1<r<\infty$ and $\norm{\epsilon}_\infty := \sup_{I\in \mathcal{D}}  \abs{\epsilon_I}  < \infty,$  then each commutator $[b,T_\epsilon^{\vec{\alpha}}]_i$ is bounded from $L^{p_1}\times L^{p_2} \times \cdots\times L^{p_m}\rightarrow L^r$  for all $1<p_1,p_2, \ldots,p_m,p < \infty$ with 
$$\sum_{j=1}^m \frac{1}{p_j} = \frac{1}{p},$$
 with estimates of the form:
$$ \norm{[b,T_\epsilon^{\vec{\alpha}}]_i(f_1,f_2,\ldots,f_m)}_p \lesssim \norm{b}_{BMO^d}\prod_{j=1}^m\norm{f_j}_{p_j}.$$

\end{itemize}

\noindent
In the above results $L^p$ stands for the Lebesgue space $L^p(\R):= \left\{f:\norm{f}_p < \infty \right\} $ with $\displaystyle\norm{f}_p = \norm{f}_{L^p} := \left(\int_\R \abs{f(x)}^p dx \right)^{1/p}.$ The Weak $L^p$ space, also denoted by $L^{p,\infty}$, is the space of all functions $f$ such that
$$ \norm{f}_{L^{p,\infty}(\mathbb{R})}:= \sup_{t>0} t \, \left\vert \{x\in \mathbb{R}: f(x) >t \} \right\vert^{1/p} < \infty.$$
Moreover, $ \displaystyle \norm{b}_{BMO^d}:=\sup_{I\in \D}\frac{1}{\abs{I}}\int_I \abs{b(x) - \La b \Ra_I} \,dx < \infty, $ is the dyadic $BMO$ norm of $b.$\\

\noindent
For the theory of linear operators, we refer to \cite{Per} and \cite{Jan}.\\

\noindent
In \cite{LOPTT}, the concept of \textit{multilinear $A_{\vec{P}}$ condition} was introduced to study the boundedness properties of the multilinear Calde$\grave{\text{r}}$on-Zygmund operators and their commutators. The use of multi(sub)linear maximal function was key to obtain the weighted estimates for the optimal range. The multilinear $A_{\vec{P}}$ condition is as follows:\\

\noindent Let $ \vec{P} = (p_1, \ldots, p_m)$ and $\vec{w} = (w_1,\ldots,w_m)$, where $1\leq p_1,\ldots,p_m <\infty$ with $\frac{1}{p_1} + \dots + \frac{1}{p_m} = \frac{1}{p},$ and $w_1, \ldots, w_m$ are non-negative measurable functions.
We say that $\vec{w}$ satisfies the multilinear $A_{\vec{P}}$ condition and we write $\vec{w} \in A_{\vec{P}}$ if
$$ \sup_I \left( \frac{1}{\abs{I}}\int_I \nu_{\vec{w}}\right)^{\frac{1}{p}} \prod_{j=1}^m\left(\frac{1}{\abs{I}}\int_I w_j ^{1-p_j'}\right)^{\frac{1}{p_j'}} < \infty,$$
where $ \displaystyle \nu_{\vec{w}} := \prod_{j=1}^m w_j^{p/p_j}$, and $\displaystyle\left(\frac{1}{\abs{I}}\int_I w_j ^{1-p_j'}\right)^{\frac{1}{p_j'}}$ is understood as $\norm{w_j^{-1}}_{L^\infty(I)}$ when $p_j =1.$\\

\noindent
We define the corresponding dyadic multilinear $A_{\vec{P}}$ class, denoted by $A_{\vec{P}}^d$, by restricting the above definition to the dyadic intervals $I$.\\ 

\noindent
We now state the main results of this article, which are the dyadic analogues of the corresponing results for the multilinear Calde$\grave{\text{r}}$on-Zygmund operators obtained in \cite{LOPTT}.\\

\noindent
\textbf{Theorem:} Let $b\in BMO^d,$ and $\epsilon =(\epsilon_I)_{I\in \D}$ be bounded.  Suppose $T \in \left\{P^{\vec{\alpha}}, T_\epsilon^{\vec{\alpha}}\right\}$ with $\vec{\alpha} \in U_m,$ or $T = \pi_b^{\vec{\alpha}}$ with $\vec{\alpha} \in \{0,1\}^m.$ Let $\vec{w} = (w_1,\ldots,w_m) \in A_{\vec{P}}^d$ for $\vec{P}=(p_1, \ldots,p_m)$ with $ \frac{1}{p_1}+\dots+ \frac{1}{p_m} = \frac{1}{p}.$
\begin{enumerate}[label = $(\alph*)$]
\item If $1 < p_1, \ldots,p_m < \infty,$ then 
$$ \norm{T(f_1,\ldots,f_m)}_{L^{p}(\nu_{\vec{w}})} \leq C \prod_{j=1}^m \norm{f_j}_{L^{p_j}(w_j)}.$$
\item If $1 \leq p_1, \ldots,p_m < \infty,$ then 
$$ \norm{T(f_1,\ldots,f_m)}_{L^{p,\infty}(\nu_{\vec{w}})} \leq C \prod_{j=1}^m \norm{f_j}_{L^{p_j}(w_j)}.$$
\end{enumerate}
\noindent
\textbf{Theorem:} Let $\vec{\alpha} \in U_m $ and $\epsilon =(\epsilon_I)_{I\in \D}$ be bounded. Suppose $b\in BMO^d$ and $\vec{w} = (w_1,\ldots,w_m) \in A_{\vec{P}}^d$ for $\vec{P}=(p_1, \ldots,p_m)$ with $ \frac{1}{p_1}+\dots+ \frac{1}{p_m} = \frac{1}{p}$ and $1 < p_1, \ldots,p_m < \infty.$ Then there exists a constant $C$ such that
$$ \left\Vert{[b,T_\epsilon^{\vec{\alpha}}]_i(f_1,\ldots,f_m)}\right\Vert_{L^{p}(\nu_{\vec{w}})} \leq C \norm{b}_{BMO^d} \prod_{j=1}^m \norm{f_j}_{L^{p_j}(w_j)}.$$ \\

\noindent
In the above results $L^p(w)$ stands for the weighted Lebesgue space $L^p(\R,w):= \left\{ f:  \norm{f}_{L^p(w)} < \infty \right\}$ with $\displaystyle \norm{f}_{L^p(w)} := \left(\int_R \abs{f(x)}^p w(x) dx\right)^{1/p}.$ Moreover, the weak space $L^{p,\infty}(w)$ is the space of all functions $f$ such that
$$ \norm{f}_{L^{p,\infty(w)}}:= \sup_{t>0} t \, w \left(\{x\in \mathbb{R}: f(x) >t \}\right) ^{1/p} < \infty.$$\\

\noindent
We organize the article as follows:\\

\noindent In section 2, we present an overview of the basic terms and related facts we will be using in this article. These include the Haar system, various maximal operators, $A_p$ and multilinear $A_{\vec{P}}$ classes, and the BMO space.\\

\noindent
In section 3, we investigate boundedness properties of the multilinear dyadic operators in the weighted setting. Weighted estimates for the commutator of the multilinear Haar multiplier with a dyadic BMO function are explored in section 4.\\

\noindent\textbf{Acknowledgement:} The author would like to thank Brett Wick for suggesting him this project and providing valuable suggestions, and also for supporting him as a research assistant for the Summer semester of 2015 (NSF DMS grant \# 0955432).
 
\section{Preliminaries}

\subsection{The Haar System}
Let $\D$ denote the standard dyadic grid on $\R,$  $$\D = \{[m2^{-k}, (m+1)2^{-k}): m,k\in \mathbb{Z}\}.$$
Associated to each dyadic interval $I$ there is a Haar function $h_I$ defined by
$$h_I(x) = \frac{1}{\abs{I}^{1/2}}\left(\mathsf{1}_{I_+} - \mathsf{1}_{I_-}\right),$$
where $I_-$ and $I_+$ are the left and right halves of $I.$\\

\noindent
The collection of all Haar functions $\{h_I: I \in \D\}$ is an orthonormal basis of $L^2(\R),$ and an unconditional basis of $L^p$ for $ 1 < p < \infty.$ In fact, if a sequence $\epsilon = \{\epsilon_I\}_{I \in \mathcal{D}}$ is bounded, the operator $T_\epsilon$ defined by
$$T_\epsilon f(x) = \sum_{I \in \mathcal{D}} \epsilon_I \La f, h_I \Ra  h_I $$
is bounded in $L^p$ for all $1 < p  < \infty.$ The converse also holds. The operator $T_\epsilon$ is called the Haar multiplier with symbol $\epsilon.$ \\

\subsection{$A_p$ classes}
A weight $w$ is a non-negative locally integrable function on $\R$ such that $0< w(x) < \infty$ for almost every $x.$ Given a weight $w$ and a measurable set $E \subseteq \mathbb{R}$, the $w$-measure of $E$ is defined by
$$ w(E) =\int_E w(x) dx.$$\\
\noindent
We say that a weight $w$ belongs to the class $A_p$ for $1 < p < \infty$ if it satisfies the Muckenhoupt condition: 
$$\sup_I \left( \frac{1}{\abs{I}}\int_I {w}\right)\left( \frac{1}{\abs{I}}\int_I {w^{-\frac{1}{p-1}}}\right)^{p-1} < \infty,$$
where the supremum is taken over all intervals. The expression on the left is called the $A_p$ (Muckenhoupt) characteristic constant of $w$, and is denoted by $[w]_{A_p}$. Note that if $p'$ is the conjugate index of $p$, i.e. $\frac{1}{p} + \frac{1}{p'} = 1$, then $ 1- p' = -\frac{1}{p-1} = -\frac{p'}{p}.$ So,
\begin{eqnarray*}
[w]_{A_p} &=& \sup_I \left( \frac{1}{\abs{I}}\int_I {w}\right)\left( \frac{1}{\abs{I}}\int_I {w^{1- p'}}\right)^{1/p'}\\
&=& \sup_I \left( \frac{1}{\abs{I}}\int_I {w}\right)\left( \frac{1}{\abs{I}}\int_I {w^{-\frac{p'}{p}}}\right)^{\frac{p}{p'}}.
\end{eqnarray*}
It can be shown that $\displaystyle \lim_{p\rightarrow 1} \left( \frac{1}{\abs{I}}\int_I {w^{-\frac{1}{p-1}}}\right)^{p-1} = \norm{w^{-1}}_{L^\infty(I)}.$ This leads to the following definition of $A_1$ class:\\
A weight $w$ is called an $A_1$ weight if $$ [w]_{A_1} := \sup_I \left( \frac{1}{\abs{I}}\int_I {w}\right)\norm{w^{-1}}_{L^\infty(I)} < \infty.$$ 
Thus $[w]_{A_1}$ is the infimum of all constants $C$ such that for all intervals $I$,
$$\frac{1}{\abs{I}}\int_I {w}\leq C w(x) \quad \text{ for a.e. } x \in I.$$
The $A_p$ classes are increasing with respect to $p$, i.e. for $1\leq p_1 < p_2 < \infty,$ $$[w]_{A_{p_2}} \leq [w]_{A_{p_1}}.$$
It is natural to define the $A_\infty$ class of weights by $$A_\infty = \displaystyle \bigcup_{p>1} A_p,$$ with 
$[w]_{A_\infty} = \inf \{ [w]_{A_p}: w \in A_p\}.$ 

\noindent
For $1\leq p <\infty$, the dyadic $A_p^d$ classes are defined by the same inequalities restricted to the dyadic intervals. Moreover, $A_\infty^d = \displaystyle \bigcup_{p>1} A_p^d.$

\subsection{Multilinear $A_{\vec{P}}$ condition} We recall the multilinear $A_{\vec{P}}$ condition introduced by Lerner et al. \cite{LOPTT}.\\

\noindent Let $ \vec{P} = (p_1, \ldots, p_m)$ and $\vec{w} = (w_1,\ldots,w_m)$, where $1\leq p_1,\ldots,p_m <\infty$ and $w_1, \ldots, w_m$ are non-negative measurable functions. Let $\frac{1}{p_1} + \dots + \frac{1}{p_m} = \frac{1}{p}.$\\
We say that $\vec{w}$ satisfies the multilinear $A_{\vec{P}}$ condition and we write $\vec{w} \in A_{\vec{P}}$ if
$$ \sup_I \left( \frac{1}{\abs{I}}\int_I \nu_{\vec{w}}\right)^{\frac{1}{p}} \prod_{j=1}^m\left(\frac{1}{\abs{I}}\int_I w_j ^{1-p_j'}\right)^{\frac{1}{p_j'}} < \infty,$$
where $ \displaystyle \nu_{\vec{w}} := \prod_{j=1}^m w_j^{p/p_j}$, and $\displaystyle\left(\frac{1}{\abs{I}}\int_I w_j ^{1-p_j'}\right)^{\frac{1}{p_j'}}$ is understood as $\norm{w_j^{-1}}_{L^\infty(I)}$ when $p_j =1.$\\

\noindent
Using H$\ddot{\text{o}}$lder's inequality, it is easy to see that $$ \prod_{j=1}^m A_{p_j} \subset A_{\vec{P}}.$$ Moreover, if $\vec{w} \in A_{\vec{P}},$ $\nu_{\vec{w}} \in A_{mp}.$ We will denote the \textit{dyadic multiliner $A_{\vec{P}}$ class } by $A_{\vec{P}}^d.$

\subsection {Maximal Operators} 
Given a function $f$, the maximal function $Mf$ is defined by
$$Mf(x):= \sup_{I\ni x} \frac{1}{\abs{I}} \int_I \abs{f(t)}\,dt,$$ 
where the supremum is taken over all intervals $I$ in $\R$ that contain $x$.\\
For $\delta >0,$ the maximal operator $M_\delta$ is defined by
$$M_\delta f(x):= M(\abs{f}^\delta)^{1/\delta}(x) = \left(\sup_{I\ni x}  \frac{1}{\abs{I}} \int_I \abs{f(t)}^\delta\,dt\right)^{1/\delta}.$$
The sharp maximal function $M^\#$ is given by
$$M^\#f(x):= \sup_{I\ni x} \inf_c \frac{1}{\abs{I}} \int_I \abs{f(t) - c}\,dt.$$
In fact, $$M^\#f(x) \approx \sup_{I\ni x} \frac{1}{\abs{I}} \int_I \abs{f(t) - \La f \Ra_I}\,dt,$$
where $\displaystyle \La f \Ra_I := \frac{1}{\abs{I}}\int_I f(t) dt$ is the average of $f$ over $I.$\\

\noindent
Given $\vec{f} = (f_1,\ldots,f_m),$ the maximal operators $\mathcal{M}$ and $\mathcal{M}_r$ with $r>0$ are defined by
$$ \mathcal{M}(\vec{f})(x) = \sup_{I \ni x} \prod_{i=1}^m \frac{1}{\abs{I}}\int_I \abs{f_i(y_i)} dy_i $$
and
$$\mathcal{M}_r(\vec{f})(x) = \sup_{I \ni x} \prod_{i=1}^m \left(\frac{1}{\abs{I}}\int_I \abs{f_i(y_i)}^r dy_i \right)^{1/r}.$$
We will be using dyadic versions of the above maximal operators which are defined by taking supremum over all dyadic intervals $I \ni x,$ instead of all intervals $I\ni x$. For convenience, we will use the same notation to denote the dyadic counterparts.\\

\noindent
We will use the following results regarding maximal functions. The dyadic analogs of these statements are also true.\\

\begin{itemize}
\item For any locally integrable function $f$, $\abs{f(x)} \leq Mf(x)$ almost everywhere. This inequality is a consequence of Lebesgue differentiation theorem and can be found in any standard Fourier Analysis textbooks, see for example \cite{JD} or \cite{LGC}. In fact, for any $\delta > 0,$ if $f \in L_{loc}^\delta(\R),$ then $\abs{f(x)} \leq M_\delta f(x)$ almost everywhere.\\

	\item For $ 0 <  \delta_1 < \delta_2 < \infty,$ $M_{\delta_1}f (x) \leq M_{\delta_2}f (x).$ This simple inequality can be verified just by using H$\ddot{\text{o}}$lder's inequality.\\ 
	
	\item For $w \in A_p$ with $1 < p < \infty$ there exists a constant $C$ such that  $$\norm{Mf}_{L^p(w)} \leq C \norm{f}_{L^p(w)}.\quad \quad 	\text{(See \cite{Per}, \cite{JD})}$$\\
	
			\item Fefferman-Stein's inequalities (see \cite{FS}): Let $w \in A_\infty$ and $0 < \delta, p <\infty.$ Then there exists a constant $C_1$ such that
	\begin{equation} \norm{M_\delta f}_{L^p(w)} \leq C_1 \norm{M_\delta^\# f}_{L^p(w)} \label{FSS}\end{equation}
	for all functions $f$ for which the left-hand side is finite.\\
	Similarly, there exists a constant $C_2$ such that \begin{equation} \norm{M_\delta f}_{L^{p,\infty}(w)} \leq C_2 \norm{M_\delta^\# f}_{L^{p,\infty}(w)} \label{FSW}\end{equation}
	for all functions $f$ for which the left-hand side is finite.\\
	
	\item Let $ \vec{P} = (p_1, \ldots, p_m)$ and $\vec{w} = (w_1,\ldots,w_m)$, where $1 < p_1,\ldots,p_m <\infty$ with $\frac{1}{p_1} + \dots + \frac{1}{p_m} = \frac{1}{p},$ and $w_1, \ldots, w_m$ are weights. Then the inequality
	\begin{equation} \norm{\mathcal{M}(\vec{f})}_{L^p(\nu_{\vec{w}})} \leq C \prod_{j=1}^m \norm{f_j}_{L^{p_j}(w_j)}\end{equation}
	holds for every $\vec{f} =(f_1, \ldots, f_m)$ if and only if $ \vec{w} \in A_{\vec{P}}.$ For $1 \leq p_1,\ldots,p_m <\infty$, the same statement is true with the inequality 
	\begin{equation} \norm{\mathcal{M}(\vec{f})}_{L^{p,\infty}(\nu_{\vec{w}})} \leq C \prod_{j=1}^m \norm{f_j}_{L^{p_j}(w_j)}.\end{equation}\\
	\noindent	
	These estimates and the one below have been obtained in \cite{LOPTT}.
	
	\item If $\vec{w} = (w_1,\ldots,w_m) \in A_{\vec{P}},$ for $ \vec{P} = (p_1, \ldots, p_m)$ with $1 < p_1,\ldots,p_m <\infty$ and $\frac{1}{p_1} + \dots + \frac{1}{p_m} = \frac{1}{p},$ then there exists an $r > 1$ such that $\vec{w}  \in A_{\vec{P}/r},$ and that \begin{equation} \label{Mr} \norm{\mathcal{M}_r(\vec{f})}_{L^p(\nu_{\vec{w}})} \leq C \prod_{j=1}^m \norm{f_j}_{L^{p_j}(w_j)}.\end{equation}\\

\end{itemize}
\subsection{BMO Space} A locally integrable function $b$ is said to be of bounded mean oscillation if
$$\norm{b}_{BMO}:=\sup_{I}\frac{1}{\abs{I}}\int_I \abs{b(x) - \La b \Ra_I} \,dx < \infty, $$
where the supremum is taken over all intervals in $\mathbb{R}.$ The space of all functions of bounded mean oscillation is denoted by $BMO.$\\

\noindent
If we take the supremum over all dyadic intervals in $\mathbb{R},$ we get a larger space of dyadic BMO functions which we denote by $BMO^d.$\\

\noindent
For $0<r<\infty,$ define
$$ BMO_r = \left\{b \in L_{loc}^r(\mathbb{R}): \norm{b}_{BMO_r} < \infty \right\},$$
where, $\displaystyle \norm{b}_{BMO_r} := \left(\sup_{I}\frac{1}{\abs{I}}\int_I \abs{b(x) - \La b \Ra_I}^r \,dx \right)^{1/r}.$\\

\noindent For any $0<r<\infty,$ the norms $\norm{b}_{BMO_r}$ and $\norm{b}_{BMO}$ are equivalent. The equivalence of norms for $r > 1$ is well-known and follows from John-Nirenberg's lemma (see \cite{JN}), while the equivalence for $0<r<1$ has been proved by Hanks in \cite{HR}. (See also \cite{SE}, page 179.)\\

\noindent
For $r=2$, it follows from the orthogonality of Haar system that 
$$ \norm{b}_{BMO_2^d} = \left(\sup_{I \in \D} \frac{1}{\abs{I}} \sum_{J \subseteq I} \abs{\widehat{b}(J)}^2\right)^{1/2}.$$

\section{Multilinear Dyadic Paraproducts and Haar Multipliers}

\noindent
We first recall the definitions of multilinear paraproduct operators and Haar multipliers introduced in \cite{IJK}.\\
For $m \geq 2$ and $\vec{\alpha} = (\alpha_1, \alpha_2, \ldots, \alpha_m) \in \{0,1\}^m$, the \textit{paraproduct operator} $P^{\vec{\alpha}}$ is defined by 
$$ P^{\vec{\alpha}}(f_1,f_2,\ldots,f_m) = \sum_{I\in\mathcal{D}} \left(\prod_{j=1}^m f_j(I,\alpha_j)\right) h_I^{\sigma(\vec{\alpha})} $$
where $f_i(I,0) = \langle f_i, h_I \rangle$, $f_i(I,1) = \langle f_i \rangle_I$ and $\sigma(\vec{\alpha}) = \#\{i: \alpha_i = 0\}.$

\noindent
Observe that if $\vec{\beta} = (\beta_1, \beta_2, \ldots,\beta_m)$ is some permutation of $\vec{\alpha} = (\alpha_1, \alpha_2, \ldots, \alpha_m)$ and $(g_1, g_2, \ldots, g_m)$ is the corresponding permutation of $(f_1, f_2, \ldots, f_m)$, then
$$P^{\vec{\alpha}} (f_1, f_2, \ldots, f_m) = P^{\vec{\beta}} (g_1, g_2, \ldots, g_m).$$

\noindent
For a given function $b,$ the paraproduct operator $\pi_b^{\vec{\alpha}}$ is defined by
$$\pi_b^{\vec{\alpha}}(f_1, f_2, \ldots, f_m) = P^{(0,\vec{\alpha})}(b,f_1, f_2, \ldots, f_m) = \sum_{I \in \mathcal{D}} \La b , h_I \Ra \prod_{j=1}^m f_j(I, \alpha_j) \; h_I^{1+\sigma(\vec{\alpha})}$$
where $(0,\vec{\alpha}) = (0,\alpha_1,\ldots, \alpha_m) \in \{0,1\}^{m+1}.$\\

\noindent Note that $$\pi_b^1(f) = P^{(0,1)}(b,f) = \sum_{I \in \mathcal{D}} b(I,0) f(I,1) h_I = \sum_{I \in \mathcal{D}} \La b, h_I \Ra \La f \Ra_I h_I = \pi_b(f).$$

\noindent
Given a symbol sequence $\epsilon = \{\epsilon_I\}_{I\in\D},$ the \textit{m-linear Haar multiplier} $ T_\epsilon^{\vec{\alpha}}$ is defined by
$$ T_\epsilon^{\vec{\alpha}} (f_1,f_2, \ldots,f_m) := \sum_{I\in \D} \epsilon_I \prod_{j=1}^m f_j(I,\alpha_j) h_I^{\sigma(\vec{\alpha})}.$$

\noindent
Let $1\leq i \leq m$. For a given function $g$, we define
$$ M_g^i(f_1, \ldots,f_m) := (f_1, \ldots,gf_i, \ldots, f_m).$$

\noindent
Note that if $T$ is multilinear, so is $T(M_g^i),$ and for $g=1,$ $ T(M_g^i) = T.$\\

\noindent
The following property of the multilinear dyadic operators will be very useful for our purpose.\\

\begin{lm} \label{DC} Let  $\vec{\alpha} = (\alpha_1,\alpha_2, \ldots,\alpha_m) \in \{0,1\}^m,$ and let $T$ be any of the $m-$linear operators $P^{\vec{\alpha}}, \pi_b^{\vec{\alpha}}$ or $T_\epsilon^{\vec{\alpha}}.$  Then for a given function $g$ and $ J \in \D,$ the function $$T\left(M_g^i(f_1,f_2, \ldots, f_m)\right) - T\left(M_g^i(f_1\mathsf{1}_J,f_2\mathsf{1}_J, \ldots, f_m \mathsf{1}_J)\right)$$   is  constant on $J.$ 
In particular, $$T(f_1,f_2, \ldots, f_m) - T(f_1\mathsf{1}_J,f_2\mathsf{1}_J, \ldots, f_m \mathsf{1}_J)$$   is  constant on $J.$\\
\end{lm}
\begin{proof} Fix $J\in \D$. Let $f_i\mathsf{1}_J = f_i^0$ and $f_i - f_i\mathsf{1}_J = f_i^\infty.$\\
Since $T(M_g^i)$ is multilinear, 
\begin{eqnarray*}
T\left(M_g^i(f_1,f_2, \ldots, f_m)\right) &=& T\left(M_g^i(f_1^0 + f_1^\infty,f_2^0 +f_2^\infty, \ldots, f_m^0+f_m^\infty)\right)\\
&=& T\left(M_g^i(f_1^0,f_2^0, \ldots, f_m^0)\right) + \sum_{\substack{\vec{\beta} \in \{0,\infty\}^m \\ \vec{\beta} \neq \vec{0}}}{T\left(M_g^i(f_1^{\beta_1},f_2^{\beta_2}, \ldots, f_m^{\beta_m})\right)},
\end{eqnarray*}
where $\vec{\beta} = (\beta_1,\ldots,\beta_m).$\\
Observe that if $I \subseteq J,$ $\widehat{f_j^\infty}(I) = \widehat{gf_j^\infty}(I)= \La f_j^\infty \Ra_I = \La gf_j^\infty \Ra_I = 0,$ since each of the functions $f_j^\infty, gf_j^\infty$ is identically 0 on $J.$ So for $\vec{\beta} \neq \vec{0},$
$$T\left(M_g^i(f_1^{\beta_1},f_2^{\beta_2}, \ldots, f_m^{\beta_m})\right) = \sum_{I\in\D} \delta_J^T \prod_{j=1}^m F_j^{\beta_j}(I,\alpha_j) h_I^{\sigma(\vec{\alpha}, T)} = \sum_{I:I \not\subseteq J}\delta_J^T \prod_{j=1}^m F_j^{\beta_j}(I,\alpha_j) h_I^{\sigma(\vec{\alpha}, T)},$$
where
$$\delta_J^T = 
     \begin{cases}
       1, &\text{if }T = P^{\vec{\alpha}}\\
       \widehat{b}(J), \;\;\; & \text{if } T = \pi_b^{\vec{\alpha}}\\
			\epsilon_J & \text{if }T = T_\epsilon^{\vec{\alpha}}
            \end{cases},$$
$$F_j^{\beta_j} = 
     \begin{cases}
       f_j^{\beta_j}, &\text{if } j \neq i\\
       gf_j^{\beta_j}, & \text{if } j=i
            \end{cases},$$
						and
$$\sigma(\vec{\alpha}, T) = 
     \begin{cases}
       \sigma(\vec{\alpha}), &\text{if }T = P^{\vec{\alpha}} \text{ or } T_\epsilon^{\vec{\alpha}}\\
       \sigma(\vec{\alpha}) + 1, & \text{if }T = \pi_b^{\vec{\alpha}}
			            \end{cases}.$$
Since each $h_I$ with $I \not\subseteq J$ is constant on $J,$ so is $T\left(M_g^i(f_1^{\beta_1},f_2^{\beta_2}, \ldots, f_m^{\beta_m})\right)$ for $\vec{\beta} \neq \vec{0}.$ Consequently, $\displaystyle\sum_{\substack{\vec{\beta} \in \{0,\infty\}^m \\ \vec{\beta} \neq \vec{0}}}{T\left(M_g^i(f_1^{\beta_1},f_2^{\beta_2}, \ldots, f_m^{\beta_m})\right)}$ is constant on $J,$ say $C_J.$  Then for every $x \in J,$ $$T\left(M_g^i(f_1,f_2, \ldots, f_m)\right)(x) - T\left(M_g^i(T)(f_1\mathsf{1}_J,f_2\mathsf{1}_J, \ldots, f_m \mathsf{1}_J)\right)(x) = c_J.$$

\noindent Taking $g =1,$ we see that $T(f_1,f_2, \ldots, f_m) - T(f_1\mathsf{1}_J,f_2\mathsf{1}_J, \ldots, f_m \mathsf{1}_J)$  is  constant on $J.$
\end{proof}
\begin{lm} \label{MF}
Let $b\in BMO^d,$ and $\epsilon =(\epsilon_I)_{I\in \D}$ be bounded. Let $T \in \left\{P^{\vec{\alpha}}, T_\epsilon^{\vec{\alpha}}\right\}$ with $\vec{\alpha} \in U_m,$ or $T = \pi_b^{\vec{\alpha}}$ with $\vec{\alpha} \in \{0,1\}^m.$  Then for $0 < \delta < \frac{1}{m},$ and $\vec{f} = (f_1,f_2,\ldots,f_m) \in L^{p_1} \times L^{p_2} \times \cdots \times L^{p_m}$ with $1\leq p_i < \infty,$ we have $$M_\delta^\#\left(T(\vec{f\,})\right)(x) \lesssim \mathcal{M}(\vec{f\,})(x).$$

\end{lm}
\begin{proof}
Fix a point $x.$ We will show that for every dyadic interval $I$ containing $x$, there exists a constant $c_I$ such that
$$ \left(\frac{1}{\abs{I}} \int_I\left| \left|T(\vec{f})(y)\right|^\delta -\left|c_I\right|^\delta \right|dy \right)^{1/\delta} \lesssim \mathcal{M}(\vec{f\,})(x),$$ 
from which the assertion follows. In fact, since $\left| \left|T(\vec{f})(y)\right|^\delta -\left|c_I\right|^\delta \right| \leq \left| T(\vec{f})(y) -c_I\right|^\delta$ for $0<\delta<1$, it suffices to show that 
$$ \left(\frac{1}{\abs{I}} \int_I\left| T(\vec{f})(y) -c_I\right|^\delta \right)^{1/\delta} \lesssim \mathcal{M}(\vec{f\,})(x).$$
Fix a dyadic interval $I$ that contains $x$, and let $f_i^0 = f\mathsf{1}_I, f_i^\infty = f_i - f_i^0.$\\

\noindent
Writing $\vec{f^0} = (f_i^0, \ldots, f_m^0),$  Lemma \ref{DC} says that  $T(\vec{f})(y) - T(\vec{f^0})(y)$ is constant for all $y$ in $I$, say $c_I.$ We then have  $T(\vec{f})(y) -c_I = T(\vec{f^0})(y)$ for all $y\in I.$ So,
$$ \left(\frac{1}{\abs{I}} \int_I\left| T(\vec{f})(y) -c_I\right|^\delta \right)^{1/\delta} =\left(\frac{1}{\abs{I}} \int_I\left| T(\vec{f^0})(y)\right|^\delta \right)^{1/\delta}.$$
We can estimate this using the following form of Kolmogorov inequality:\\
If $0<p<q<\infty,$ then for any measurable function $f,$ there exists a constant $C = C(p,q)$ such that
\begin{equation} \left \Vert f \right \Vert_{L^p\left(I, \frac{dy}{\abs{I}}\right)} \leq C \left \Vert f \right \Vert_{L^{q,\infty}\left(I, \frac{dy}{\abs{I}}\right)}.
\label{eq:Kolmogorov}
\end{equation}
For $p = \delta, q = 1/m$ and $f = T(\vec{f^0}),$ \eqref{eq:Kolmogorov} becomes
$$\left(\frac{1}{\abs{I}} \int_I\left| T(\vec{f^0})(y)\right|^\delta dy\right)^{1/\delta} \leq C \left \Vert T(\vec{f^0})(y) \right \Vert_{L^{1/m,\infty}\left(I, \frac{dy}{\abs{I}}\right)}.$$
Now,\begin{eqnarray*} \left \Vert T(\vec{f^0})(y) \right \Vert_{L^{1/m,\infty}\left(I, \frac{dy}{\abs{I}}\right)}
 &=& \sup_{t>0}  t\left( \frac{1}{\abs{I}}\left\vert \left\{y \in I: \left\vert T(\vec{f^0})(y) \right\vert > t \right\}\right\vert \right)^m\\
&\leq& \sup_{t>0}  \frac{t}{\abs{I}^m} \left\vert \left\{y: \frac{1}{\abs{I}^m}\left\vert T(\vec{f^0})(y) \right\vert > \frac{t}{\abs{I}^m} \right\}\right\vert^m\\
&=& \sup_{t>0}  \frac{t}{\abs{I}^m} \left\vert \left\{y: \left\vert T\left(\frac{f_1^0}{\abs{I}}, \ldots, \frac{f_m^0}{\abs{I}} \right)(y)\right\vert > \frac{t}{\abs{I}^m} \right\}\right\vert^m\\
&=& \left \Vert T\left(\frac{f_1^0}{\abs{I}}, \ldots, \frac{f_m^0}{\abs{I}} \right)(y) \right \Vert_{L^{1/m,\infty}}.
\end{eqnarray*}
\noindent
Since  $\displaystyle \frac{f_i^0}{\abs{I}} \in L^1$ for all $1 \leq i \leq m$, it follows from the boundedness of $T: L^1 \times\cdots \times L^1 \rightarrow L^{1/m,\infty}$ that
\begin{eqnarray*}
\left \Vert T\left(\frac{f_1^0}{\abs{I}}, \ldots, \frac{f_m^0}{\abs{I}} \right)(y) \right \Vert_{L^{1/m,\infty}}
 &\lesssim & \prod_{i=1}^m \left\Vert \frac{f_i^0}{\abs{I}} \right\Vert_{L^1}\\
&=& \prod_{i=1}^m \int \frac{\abs{f_i^0}}{\abs{I}}\\
&=& \prod_{i=1}^m \frac{1}{\abs{I}}\int_I {\abs{f_i}}\\
&\leq & \mathcal{M}(\vec{f\,})(x).
\end{eqnarray*}
This completes the proof.
\end{proof}

\noindent
The following lemma gives us the finiteness condition needed to apply Fefferman-Stein inequalities \ref{FSS} and \ref{FSW} for the multilinear dyadic operators.\\

\begin{lm}\label{MFD}
Let $w \in A_\infty^d$ and $\vec{f} = (f_1, \ldots,f_m)$ where each $f_i$ is bounded and has compact support. If $\left\Vert \mathcal{M}(\vec{f\,})\right\Vert_{L^p(w)} < \infty$ for some $p > 0,$ then there exists a $\delta \in (0, 1/m)$ such that $\left\Vert {M}_\delta\left(T(\vec{f\,})\right)\right\Vert_{L^p(w)} < \infty.$ Similarly, if $\left\Vert \mathcal{M}(\vec{f\,})\right\Vert_{L^{p,\infty}(w)} < \infty$ for some $p > 0,$ then there exists a $\delta \in (0, 1/m)$ such that $\left\Vert {M}_\delta\left(T(\vec{f\,})\right)\right\Vert_{L^{p,\infty}(w)} < \infty.$
\end{lm}
\begin{proof} We prove the first assertion, the second one follows from similar arguments.\\

\noindent
Since $w \in A_\infty^d,$ it is in $A_{p_0}^d$ for some $p_0 > \text{max} (1, pm).$ Then for any $\delta$ with $0<\delta<p/p_0 < 1/m,$ we have\\
\begin{eqnarray*}
\left\Vert M_\delta \left(T(\vec{f}\,)\right) \right\Vert_{L^p(w)} 
&\leq & \left\Vert M_{p/p_0} \left(T(\vec{f}\,)\right) \right\Vert_{L^p(w)}\\
&=& \left[\int_{\mathbb{R}} \left \{\left(\sup_{I\ni x} \frac{1}{\abs{I}} \int_I \abs{T(\vec{f}\,)}^{p/p_0}\,dt \right)^{p_0/p} \right \}^p dw(x)\right]^{1/p}\\
&=& \left[\int_{\mathbb{R}} M\left(T(\vec{f}\,)^{p/p_0}\right)^{p_0}  dw\right]^{\frac{1}{p_0}\times \frac{p_0}{p}}\\
&=& \left\Vert M \left(T(\vec{f}\,)^{p/p_0}\right) \right\Vert_{L^{p_0}(w)} ^ {p_0/p},
\end{eqnarray*}
The boundedness of $M: L^{p_0}(w) \rightarrow L^{p_0}(w)$ for $w \in A_{p_0}^d$ gives 
$$\left\Vert M \left(T(\vec{f}\,)^{p/p_0}\right) \right\Vert_{L^{p_0}(w)}  \lesssim \left\Vert T(\vec{f}\,)^{p/p_0} \right\Vert_{L^{p_0}(w)}.$$
Consequently,
\begin{eqnarray*}
\left\Vert M_\delta \left(T(\vec{f}\,)\right) \right\Vert_{L^p(w)} 
&\lesssim & \left\Vert T(\vec{f}\,)^{p/p_0} \right\Vert_{L^{p_0}(w)} ^ {p_0/p}\\
&=& \left(\int_{\mathbb{R}} \left \vert T(\vec{f}\,)^{p/p_0} \right \vert ^{p_0} dw\right)^{\frac{1}{p_0}\times \frac{p_0}{p}}\\
&=& \left(\int_{\mathbb{R}} \left\vert T(\vec{f}\,)\right\vert^{p}  dw\right)^{1/p}\\
&=& \left\Vert T(\vec{f}\,) \right\Vert_{L^p(w)},
\end{eqnarray*}
So, it suffices to prove that $\left\Vert T(\vec{f\,})\right\Vert_{L^p(w)} < \infty.$\\

\noindent
Since each $f_i$ has compact support, there exist dyadic intervals $S' = [0, 2^{-k})$ and $S'' = [-2^{-k}, 0)$ such that the support of every $f_i$ is contained in $S=S'\cup S''.$\\
To prove the assertion, it suffices to show that $$\left\Vert T(\vec{f\,})\right\Vert_{L^p(S,w)} < \infty \;\text{ and }\; \left\Vert T(\vec{f\,})\right\Vert_{L^p(\mathbb{R}\backslash S, w)} < \infty.$$

\noindent
Since $w \in A_\infty^d,$ $w^{1+\gamma} \in L^1_{loc}$ for sufficiently small $\gamma$, (see \cite{Per}  or \cite{LGM}). In particular,  $w \in L^q(S)$ for $q:= 1+\gamma.$ We can choose $\gamma$ small enough so that $w \in L^q(S)$ and $q'p > \frac{1}{m}.$ Then by H$\ddot{\text{o}}$lder's inequality, we have
\begin{eqnarray*}
\left\Vert T(\vec{f\,})\right\Vert_{L^p(S,w)} &=& \left(\int_S \left\vert T(\vec{f\,})\right\vert^p  w dx\right)^{1/p}\\
&\leq& \left(\left(\int_S \left\vert T(\vec{f\,})\right\vert^{pq'}  dx\right)^{1/q'} \left(\int_S w^q dx\right)^{1/q}\right)^{1/p}\\
&<& \infty.
\end{eqnarray*}
Here, the finiteness of $\displaystyle\int_S \left\vert T(\vec{f\,})\right\vert^{pq'}  dx$ follows from the boundedness of $T: L^{mpq'}\times \cdots \times  L^{mpq'} \rightarrow  L^{pq'},$ and the fact that each $f_i$ (being bounded with compact support) is in $ L^{mpq'}.$ We refer to \cite{IJK} for the unweighted theory of multilinear dyadic operators.\\

\noindent
To prove $\left\Vert T(\vec{f\,})\right\Vert_{L^p(\mathbb{R}\backslash S, w)} < \infty,$ it suffices to show that
$$ \left\vert T(\vec{f\,})(x) \right\vert \leq C \mathcal{M}(\vec{f\,})(x)\quad \text{ for every } x\in \mathbb{R}\backslash S.$$
\noindent
We prove this for $T = \pi_b^{\vec{\alpha}}$. Proofs for $P^{\vec{\alpha}}$ and $T_\epsilon^{\vec{\alpha}}$ follow similarly.\\
\noindent
Fix $x \in \mathbb{R}\backslash S.$ Let $I_x$ be the smallest dyadic interval that contains $x$ and one of the intervals $S'$ and $S''.$

\noindent
For definiteness, assume $x >0.$ In this case $I_x$ is the smallest dyadic interval containing $x$ and $S'.$ Note that if $x \notin I, h_I(x) = 0$ and, if $x \in I$ with $I \cap S' = \emptyset, f_j(I,\alpha_j) = 0$ for each $j.$ So,
\begin{eqnarray*}
\left\vert \pi_b^{\vec{\alpha}}(\vec{f\,})(x) \right\vert 
&=& \left\vert \sum_{I \in \mathcal{D}} \widehat{b}(I) \prod_{j=1}^m f_j(I, \alpha_j) \; h_I^{1+\sigma(\vec{\alpha})}(x) \right\vert\\
&=& \left\vert \sum_{I \supseteq I_x} \widehat{b}(I) \prod_{j=1}^m f_j(I, \alpha_j) \; h_I^{1+\sigma(\vec{\alpha})}(x) \right\vert\\
&\leq& \sum_{I \supseteq I_x} \frac{\left\vert\widehat{b}(I)\right\vert}{\sqrt{\abs{I}}} \left(\prod_{j:\alpha_j=0} \frac{\left\vert\widehat{f_j}(I)\right\vert}{\sqrt{\abs{I}}}\right) \left(\prod_{j:\alpha_j=1} \left\vert\La f_j \Ra_I\right\vert\right) \; \mathsf{1}_I(x)\\
&\leq& \left\Vert b\right\Vert_{BMO^d}\sum_{I \supseteq I_x}  \left(\prod_{j:\alpha_j=0} \frac{\left\vert\widehat{f_j}(I)\right\vert}{\sqrt{\abs{I}}}\right) \left(\prod_{j:\alpha_j=1} \left\vert\La f_j \Ra_I\right\vert\right) ,\\
\end{eqnarray*}
where the last inequality follows from the fact that for $b \in BMO^d,$
$$ \frac{\left\vert\widehat{b}(I)\right\vert}{\sqrt{\abs{I}}} \leq \left(\frac{1}{\abs{I}}\sum_{J \subseteq I} \left\vert\widehat{b}(I)\right\vert^2 \right)^{1/2} \leq \left\Vert b\right\Vert_{BMO^d}.$$

\noindent Note that
$\displaystyle \frac{\left\vert\widehat{f_j}(I)\right\vert}{\sqrt{\abs{I}}} = \frac{1}{\sqrt{\abs{I}}} \left\vert \int f_jh_I \right\vert \leq \frac{1}{\sqrt{\abs{I}}}  \int \left\vert f_j\right\vert \frac{\mathsf{1}_I}{\sqrt{\abs{I}}} = \frac{1}{\abs{I}}\int_I \abs{f_j} = \La \abs{f_j} \Ra_I,$ and since $f_j$ is 0 on $\mathbb{R} \backslash S,$ we have
$\displaystyle \La \abs{f_j} \Ra_{I^1} = \frac{\La \abs{f_j} \Ra_I}{2}$ whenever $I^1$ is the parent of $I$ with $I_x \subseteq I.$ So, we have
\begin{eqnarray*}
\left\vert \pi_b^{\vec{\alpha}}(\vec{f\,})(x) \right\vert 
&\leq& \left\Vert b\right\Vert_{BMO^d}\sum_{I \supseteq I_x}   \prod_{j=1}^m \La \abs{f_j} \Ra_I \\
&=& \left\Vert b\right\Vert_{BMO^d} \left( \prod_{j=1}^m \La \abs{f_j} \Ra_{I_x} + \frac{1}{2^m}\prod_{j=1}^m \La \abs{f_j} \Ra_{I_x} + \frac{1}{2^{2m}}\prod_{j=1}^m \La \abs{f_j} \Ra_{I_x}+ \cdots \right)\\
&=& \frac{2^m}{(2^m -1)}\left\Vert b\right\Vert_{BMO^d} \prod_{j=1}^m \La \abs{f_j} \Ra_{I_x}\\
&\leq& \frac{2^m}{(2^m -1)}\left\Vert b\right\Vert_{BMO^d} \mathcal{M}(\vec{f\,})(x).
\end{eqnarray*}
The same proof works for $x <0$ too. This completes the proof.
\end{proof}
\begin{thm}\label{MT}
Let $b\in BMO^d,$ and $\epsilon =(\epsilon_I)_{I\in \D}$ be bounded.  Let $T \in \left\{P^{\vec{\alpha}}, T_\epsilon^{\vec{\alpha}}\right\}$ with $\vec{\alpha} \in U_m,$ or $T = \pi_b^{\vec{\alpha}}$ with $\vec{\alpha} \in \{0,1\}^m.$  Then for $w\in A_\infty^d$ and $p>0$,
$$\norm{T(\vec{f}\,)}_{L^p(w)} \lesssim \norm{\mathcal{M}(\vec{f}\,)}_{L^p(w)}$$
and $$\norm{T(\vec{f}\,)}_{L^{p,\infty}(w)} \lesssim \norm{\mathcal{M}(\vec{f}\,)}_{L^{p,\infty}(w)}$$
for all $m$-tuples $\vec{f} =(f_1,\ldots,f_m)$ of bounded functions with compact support. 
\end{thm}
\begin{proof}
To prove the first inequality, assume that $\norm{\mathcal{M}(\vec{f}\,)}_{L^p(w)} < \infty,$ otherwise there is nothing to prove. Then by Lemma \ref{MFD}, there exists a $\delta \in (0, 1/m)$ such that $\left\Vert {M}_\delta\left(T(\vec{f\,})\right)\right\Vert_{L^p(w)} < \infty.$ For such $\delta$, we have
$$ \left\Vert T(\vec{f\,})\right\Vert_{L^p(w)} \leq \left\Vert {M}_\delta\left(T(\vec{f\,})\right)\right\Vert_{L^p(w)} \leq C \left\Vert {M}_\delta^\#\left(T(\vec{f\,})\right)\right\Vert_{L^p(w)} \leq C\left\Vert {\mathcal{M}}(\vec{f\,})\right\Vert_{L^p(w)},$$ where the first and last inequalities follow from pointwise control and the second inequality is the Fefferman-Stein's inequality \eqref{FSS}.\\

\noindent
Proof of the second inequality follows similarly, by applying Lemma \ref{MFD} and using the Fefferman-Stein's inequality \eqref{FSW} for weak-type estimates.
\end{proof}
\begin{thm}
Let $b\in BMO^d,$ and $\epsilon =(\epsilon_I)_{I\in \D}$ be bounded.  Suppose $T \in \left\{P^{\vec{\alpha}}, T_\epsilon^{\vec{\alpha}}\right\}$ with $\vec{\alpha} \in U_m,$ or $T = \pi_b^{\vec{\alpha}}$ with $\vec{\alpha} \in \{0,1\}^m.$ Let $\vec{w} = (w_1,\ldots,w_m) \in A_{\vec{P}}^d$ for $\vec{P}=(p_1, \ldots,p_m)$ with $ \frac{1}{p_1}+\dots+ \frac{1}{p_m} = \frac{1}{p}.$
\begin{enumerate}[label = $(\alph*)$]
\item If $1 < p_1, \ldots,p_m < \infty,$ then 
\begin{equation} \norm{T(\vec{f})}_{L^{p}(\nu_{\vec{w}})} \leq C \prod_{j=1}^m \norm{f_j}_{L^{p_j}(w_j)}.\end{equation}
\item If $1 \leq p_1, \ldots,p_m < \infty,$ then 
\begin{equation} \norm{T(\vec{f})}_{L^{p,\infty}(\nu_{\vec{w}})} \leq C \prod_{j=1}^m \norm{f_j}_{L^{p_j}(w_j)}.\end{equation}
\end{enumerate}
\end{thm}
\begin{proof} Since the simple functions in $L^p(w)$ are dense in $L^p(w)$ for any weight $w$ (see \cite{BS}), it suffices to prove the estimates for $\vec{f} = (f_1,f_2,\ldots,f_m)$ with $f_i \in L^{p_i}(w_i)$ simple.
Note that $\vec{w} = (w_1,\ldots,w_m) \in A_{\vec{P}}^d$ implies that $\nu_{\vec{w}} \in A_\infty^d.$ So, by Theorem \ref{MT} and the boundedness properties of the multilinear maximal function $\mathcal{M}$, we have\\
$$\norm{T(\vec{f}\,)}_{L^p(\nu_{\vec{w}})} \lesssim \norm{\mathcal{M}(\vec{f}\,)}_{L^p(\nu_{\vec{w}})} \lesssim \prod_{j=1}^m \norm{f_j}_{L^{p_j}(w_j)},$$
and $$\norm{T(\vec{f}\,)}_{L^{p,\infty}(\nu_{\vec{w}})} \lesssim \norm{\mathcal{M}(\vec{f}\,)}_{L^{p,\infty}(\nu_{\vec{w}})}\lesssim \prod_{j=1}^m \norm{f_j}_{L^{p_j}(w_j)}.$$

\end{proof}
\section{Commutators of Multilinear Haar Multipliers}

\begin{dfn}Let $\vec{\alpha} \in U_m $ and $\epsilon =(\epsilon_I)_{I\in \D}$ be bounded.  Given a locally integrable function $b$, we define the commutator $[b,T_\epsilon^{\vec{\alpha}}]_i, 1\leq i \leq m,$ by
$$[b,T_\epsilon^{\vec{\alpha}}]_i(f_1,f_2,\ldots,f_m)(x) := b(x)T_\epsilon^{\vec{\alpha}}(f_1,f_2,\ldots,f_m)(x) - T_\epsilon^{\vec{\alpha}}(f_1, \ldots, bf_i,\ldots,f_m)(x). $$
 $$ i.e. \quad [b,T_\epsilon^{\vec{\alpha}}]_i = M_b \circ T_\epsilon^{\vec{\alpha}} - T_\epsilon^{\vec{\alpha}} \circ M_b^i.$$
\end{dfn}

\begin{thm}\label{MFC}
Let $\vec{\alpha} \in U_m $ and $\epsilon =(\epsilon_I)_{I\in \D}$ be bounded. Let $\delta \in (0, 1/m)$ and $\gamma > \delta$. Then for any $r>1,$
\begin{equation}
M_\delta^\#\left( [b,T_\epsilon^{\vec{\alpha}}]_i(\vec{f})\right)(x) \lesssim \norm{b}_{BMO^d} \left(\mathcal{M}_r(\vec{f})(x) + M_\gamma \left(T_\epsilon^{\vec{\alpha}}(\vec{f}\,) \right)(x)\right)
\end{equation}
for all $m$-tuples $\vec{f} = (f_1,f_2,\ldots,f_m)$ of bounded measurable functions with compact support.
\end{thm}
\begin{proof} Fix $x \in \R$. As in the proof of Lemma \ref{MF}, it suffices to show that for every $I \in \D$ containing $x,$ there exists a constant $C_I$ such that
$$\left(\frac{1}{\abs{I}}\int_I \left\vert [b,T_\epsilon^{\vec{\alpha}}]_i(\vec{f})(t) - C_I  \right\vert^\delta dt  \right)^{1/\delta} \lesssim \norm{b}_{BMO^d} \left(\mathcal{M}_r(\vec{f})(x) + M_\gamma \left(T_\epsilon^{\vec{\alpha}}(\vec{f}\,) \right)(x)\right).$$
Fix $I \in \D$ containing $x,$ and take $C_I = T_\epsilon^{\vec{\alpha}} \left(M_g^i(\vec{f^0})\right)(t) - T_\epsilon^{\vec{\alpha}} \left(M_g^i(\vec{f})\right)(t)$, where $g=b-\La b \Ra_I$ and $\vec{f^0} = (f_1^0,\ldots,f_m^0)$ with $f_i^0 = f_i \mathsf{1}_I.$ Lemma \ref{DC} shows that this is indeed a constant on $I.$ Since $T_\epsilon^{\vec{\alpha}}$ is multilinear,  
\begin{eqnarray*} [b,T_\epsilon^{\vec{\alpha}}]_i(\vec{f})(t) 
&=& b(t)\,T_\epsilon^{\vec{\alpha}}(\vec{f})(t) - T_\epsilon^{\vec{\alpha}}(f_1, \ldots, bf_i,\ldots,f_m)(t)\\
&=& \left(b(t) -\La b \Ra_I \right)\,T_\epsilon^{\vec{\alpha}}(\vec{f})(t) - T_\epsilon^{\vec{\alpha}}(f_1, \ldots, (b -\La b \Ra_I) f_i,\ldots,f_m)(t)\\
&=& \left(b(t) -\La b \Ra_I \right)\,T_\epsilon^{\vec{\alpha}}(\vec{f})(t) - T_\epsilon^{\vec{\alpha}}\left(M_g^i(\vec{f})\right)(t).
\end{eqnarray*}
 So,
\begin{eqnarray*}
&&\left(\frac{1}{\abs{I}}\int_I \left\vert [b,T_\epsilon^{\vec{\alpha}}]_i(\vec{f})(t) - C_I  \right\vert^\delta dt  \right)^{1/\delta}\\
&=& \left(\frac{1}{\abs{I}}\int_I \left\vert \left(b(t) -\La b \Ra_I \right)\,T_\epsilon^{\vec{\alpha}}(\vec{f})(t) - T_\epsilon^{\vec{\alpha}}\left(M_g^i(\vec{f})\right)(t) - C_I  \right\vert^\delta dt  \right)^{1/\delta}\\
&=& \left(\frac{1}{\abs{I}}\int_I \left\vert \left(b(t) -\La b \Ra_I \right)\,T_\epsilon^{\vec{\alpha}}(\vec{f})(t) - T_\epsilon^{\vec{\alpha}}\left(M_g^i(\vec{f^0})\right)(t) \right\vert^\delta dt  \right)^{1/\delta}\\
&\lesssim& \left(\frac{1}{\abs{I}}\int_I \left\vert \left(b(t) -\La b \Ra_I \right)\,T_\epsilon^{\vec{\alpha}}(\vec{f})(t) \right\vert^\delta dt  \right)^{1/\delta } + \left(\frac{1}{\abs{I}}\int_I \left\vert T_\epsilon^{\vec{\alpha}}\left(M_g^i(\vec{f^0})\right)(t)   \right\vert^\delta dt  \right)^{1/\delta }.\\
\end{eqnarray*}
Note that $ \gamma/ \delta >1.$ For any $q \in (1, \gamma/ \delta)$, H$\ddot{\text{o}}$lder's inequlity gives
\begin{eqnarray*}
&&\left(\frac{1}{\abs{I}}\int_I \left\vert \left(b(t) -\La b \Ra_I \right)\,T_\epsilon^{\vec{\alpha}}(\vec{f})(t) \right\vert^\delta dt  \right)^{1/\delta}\\
&\leq&  \left(\frac{1}{\abs{I}}\int_I \left\vert \left(b(t) -\La b \Ra_I \right) \right\vert^{\delta q'} dt  \right)^{1/\delta q'}\left(\frac{1}{\abs{I}}\int_I \left\vert T_\epsilon^{\vec{\alpha}}(\vec{f})(t) \right\vert^{\delta q} dt  \right)^{1/\delta q}\\
&\lesssim & \norm{b}_{BMO^d}\, M_{\delta q}\left(T_\epsilon^{\vec{\alpha}}(\vec{f})\right)(x)\\
&\leq& \norm{b}_{BMO^d}\, M_{\gamma}\left(T_\epsilon^{\vec{\alpha}}(\vec{f})\right)(x).
\end{eqnarray*}
As in the proof of Lemma \ref{MF}, we can apply Kolmogorov's inequality to obtain
\begin{eqnarray*}
&&\left(\frac{1}{\abs{I}}\int_I \left\vert T_\epsilon^{\vec{\alpha}}(f_1^0, \ldots, (b -\La b \Ra_I) f_i^0,\ldots,f_m^0)(t) \right\vert^\delta dt  \right)^{1/\delta}\\
&\leq&  \left\Vert T_\epsilon^{\vec{\alpha}}(f_1^0, \ldots, (b -\La b \Ra_I) f_i^0,\ldots,f_m^0)(t) \right\Vert _{L^{\frac{1}{m},\infty\left(I, \frac{dt}{\abs{I}}\right)}}\\
&\leq& \frac{1}{\abs{I}}\int_I \left\vert \left(b(t) -\La b \Ra_I \right)\,f_i^0(t) \right\vert dt \prod_{j=1, j \neq i}^m \frac{1}{\abs{I}}\int_I \left\vert f_j^0(t) \right\vert dt \\
&\leq& \left(\frac{1}{\abs{I}}\int_I \left\vert b(t) -\La b \Ra_I \right\vert ^{r'}dt \right)^{1/r'}\left(\frac{1}{\abs{I}}\int_I \left\vert f_i^0(t)\right\vert^r dt\right)^{1/r} \prod_{j=1, j \neq i}^m \left(\frac{1}{\abs{I}}\int_I \left\vert f_j^0(t) \right\vert^r dt\right)^{1/r} \\
&\lesssim& \norm{b}_{BMO^d} \prod_{j=1}^m \left(\frac{1}{\abs{I}}\int_I \left\vert f_j(t) \right\vert^r dt\right)^{1/r} \\
&\leq& \norm{b}_{BMO^d} \, \mathcal{M}_r(\vec{f})(x).
\end{eqnarray*}
We thus have $$M_\delta^\#\left( [b,T_\epsilon^{\vec{\alpha}}]_i(\vec{f})\right)(x) \lesssim \norm{b}_{BMO^d} \left(\mathcal{M}_r(\vec{f})(x) + M_\gamma \left(T_\epsilon^{\vec{\alpha}}(\vec{f}\,) \right)(x)\right)
.$$
\end{proof}

\noindent
Lemma \ref{MFD} is also true for the commutators of the multilinear Haar multipliers with a bounded function $b.$\\

\begin{lm}\label{MFDC}
Let $w \in A_\infty^d$ and $\vec{f} = (f_1, \ldots,f_m)$ where each $f_i$ is bounded and has compact support. If $\left\Vert \mathcal{M}(\vec{f\,})\right\Vert_{L^p(w)} < \infty$ for some $p > 0,$ and $b$ bounded, then there exists a $\delta \in (0, 1/m)$ such that $\left\Vert {M}_\delta\left([b,T_\epsilon^{\vec{\alpha}}]_i(\vec{f\,})\right)\right\Vert_{L^p(w)} < \infty.$ 
\end{lm}
\begin{proof}
Since each $f_i$ has compact support, there exist dyadic intervals $S' = [0, 2^{-k})$ and $S'' = [-2^{-k}, 0)$ such that the support of every $f_i$ is contained in $S=S'\cup S''.$\\
Following the arguments used in the proof of Lemma \ref{MFD}, we get
$$\left\Vert {M}_\delta\left([b,T_\epsilon^{\vec{\alpha}}]_i(\vec{f\,})\right)\right\Vert_{L^p(w)} \leq \left\Vert [b,T_\epsilon^{\vec{\alpha}}]_i(\vec{f\,})\right\Vert_{L^p(w)}.$$
So, it suffices to prove that
$$\left\Vert [b,T_\epsilon^{\vec{\alpha}}]_i(\vec{f\,})\right\Vert_{L^p(S,w)} < \infty \;\text{ and }\; \left\Vert [b,T_\epsilon^{\vec{\alpha}}]_i(\vec{f\,})
\right\Vert_{L^p(\mathbb{R}\backslash S, w)} < \infty.$$
Since $w \in A_\infty^d,$ $w^{1+\gamma} \in L^1_{loc}$ for sufficiently small $\gamma$, (see \cite{Per}  or \cite{LGM}). In particular,  $w \in L^q(S)$ for $q:= 1+\gamma.$ We can choose $\gamma$ small enough so that $w \in L^q(S)$ and $q'p > 1.$ Then by H$\ddot{\text{o}}$lder's inequality, we have
\begin{eqnarray*}
\left\Vert [b,T_\epsilon^{\vec{\alpha}}]_i(\vec{f})\right\Vert_{L^p(S,w)} &=& \left(\int_S \left\vert [b,T_\epsilon^{\vec{\alpha}}]_i(\vec{f})\right\vert^p  w dx\right)^{1/p}\\
&\leq& \left(\left(\int_S \left\vert [b,T_\epsilon^{\vec{\alpha}}]_i(\vec{f})\right\vert^{pq'}  dx\right)^{1/q'} \left(\int_S w^q dx\right)^{1/q}\right)^{1/p}\\
&<& \infty.
\end{eqnarray*}
Here, $\displaystyle\int_S w^q dx < \infty$ because $w \in L^q_{loc}$, and the finiteness of $\displaystyle\int_S \left\vert [b,T_\epsilon^{\vec{\alpha}}]_i(\vec{f})\right\vert^{pq'}  dx$ follows from boundedness of $[b,T_\epsilon^{\vec{\alpha}}]_i: L^{mpq'}\times \cdots \times  L^{mpq'} \rightarrow  L^{pq'},$ and the fact that each $f_i$ (being bounded with compact support) is in $ L^{mpq'}.$  For the unweighted theory of the commutators of multilinear Haar multipliers we refer to \cite{IJK}. Note that to prove finiteness of $\left\Vert [b,T_\epsilon^{\vec{\alpha}}]_i(\vec{f})\right\Vert_{L^p(S,w)}$ we may assume that the BMO function $b$ is in some $L^p$ space with $1<p<\infty.$ Indeed, for all $x \in S,$
$$ [b,T_\epsilon^{\vec{\alpha}}]_i(\vec{f})(x) =  [b\mathsf{1}_S,T_\epsilon^{\vec{\alpha}}]_i(\vec{f})(x),$$
for all $\vec{f} = (f_1,\ldots,f_m)$ with $f_i$ supported in $S.$\\

\noindent Now to prove $\left\Vert [b,T_\epsilon^{\vec{\alpha}}]_i(\vec{f})\right\Vert_{L^p(\mathbb{R}\backslash S, w)} < \infty,$ it suffices to show that for every $x \in \mathbb{R}\backslash S,$
$$\left\vert[b,T_\epsilon^{\vec{\alpha}}]_i(\vec{f})(x)\right\vert  \leq \mathcal{M}(\vec{f})(x).$$

Fix $x \in \mathbb{R}\backslash S.$ For definiteness, assume that $x >0,$ and let $I_x$ be the smallest dyadic interval that contains $x$ and the interval $S'.$  Note that if $x \notin I, h_I(x) = 0$ and, if $x \in I$ with $I \cap S' = \emptyset, f_j(I,\alpha_j) = 0$ for each $j.$ So,
\begin{eqnarray*}
&&\left\vert [b,T_\epsilon^{\vec{\alpha}}]_i(\vec{f})(x) \right\vert\\
&\leq& \left\vert b(x)\right\vert \left\vert T_\epsilon^{\vec{\alpha}}(f_1,f_2,\ldots,f_m)(x)\right\vert + \left\vert T_\epsilon^{\vec{\alpha}}(f_1, \ldots, bf_i,\ldots,f_m)(x)\right\vert\\ 
&=& \left\vert b(x)\right\vert \left\vert \sum_{I \supseteq I_x} \epsilon_I \prod_{j=1}^m f_j(I, \alpha_j) \; h_I^{\sigma(\vec{\alpha})}(x) \right\vert + \left\vert \sum_{I \supseteq I_x} \epsilon_I (bf_i)(I,\alpha_i)\prod_{\substack{j=1\\ j\neq i}}^m f_j(I, \alpha_j) \; h_I^{\sigma(\vec{\alpha})}(x) \right\vert\\
&\leq& \abs{b(x)} \sum_{I \supseteq I_x} \abs{\epsilon_I} \left(\prod_{j:\alpha_j=0} \frac{\left\vert\widehat{f_j}(I)\right\vert}{\sqrt{\abs{I}}}\right) \left(\prod_{j:\alpha_j=1} \left\vert\La f_j \Ra_I\right\vert\right) \; \mathsf{1}_I(x)\\
&& \quad \quad+ \abs{b(x)} \sum_{I \supseteq I_x} \abs{\epsilon_I} \abs{(bf_i)(I,\alpha_i)} \left(\prod_{\substack{j:\alpha_j=0 \\ j\neq i}} \frac{\left\vert\widehat{f_j}(I)\right\vert}{\sqrt{\abs{I}}}\right) \left(\prod_{\substack{j:\alpha_j=1\\ j\neq i}} \left\vert\La f_j \Ra_I\right\vert\right) \; \mathsf{1}_I(x)\\
\end{eqnarray*}

\noindent
We have $\displaystyle \frac{\left\vert\widehat{f_j}(I)\right\vert}{\sqrt{\abs{I}}} = \frac{1}{\sqrt{\abs{I}}} \left\vert \int f_jh_I \right\vert \leq \frac{1}{\sqrt{\abs{I}}}  \int \left\vert f_j\right\vert \frac{\mathsf{1}_I}{\sqrt{\abs{I}}} = \frac{1}{\abs{I}}\int_I \abs{f_j} = \La \abs{f_j} \Ra_I.$ Since $f_j$ is 0 on $\mathbb{R} \backslash S,$
$\displaystyle \La \abs{f_j} \Ra_{I^1} = \frac{\La \abs{f_j} \Ra_I}{2}$ whenever $I^1$ is the parent of $I$ with $I_x \subseteq I.$  Moreover, $\left\vert (bf_i)(I,\alpha_i)\right\vert \leq \abs{b} \La \abs{f_i}\Ra_I$. So,

\begin{eqnarray*}
&&\left\vert [b,T_\epsilon^{\vec{\alpha}}]_i(\vec{f})(x) \right\vert\\
&\leq& 2 \left(\sup_{x\in \R}\abs{b(x)}\right) \left( \sup_{I\in \D} \abs{\epsilon_I}\right)  \sum_{I \supseteq I_x}  \left(\prod_{j:\alpha_j=0} \frac{\left\vert\widehat{f_j}(I)\right\vert}{\sqrt{\abs{I}}}\right) \left(\prod_{j:\alpha_j=1} \left\vert\La f_j \Ra_I\right\vert\right).\\
&\leq& 2 \left(\sup_{x\in \R}\abs{b(x)}\right) \left( \sup_{I\in \D} \abs{\epsilon_I}\right)\sum_{I \supseteq I_x}   \prod_{j=1}^m \La \abs{f_j} \Ra_I \\
&=& 2 \left(\sup_{x\in \R}\abs{b(x)}\right) \left( \sup_{I\in \D} \abs{\epsilon_I}\right) \left( \prod_{j=1}^m \La \abs{f_j} \Ra_{I_x} + \frac{1}{2^m}\prod_{j=1}^m \La \abs{f_j} \Ra_{I_x} + \frac{1}{2^{2m}}\prod_{j=1}^m \La \abs{f_j} \Ra_{I_x}+ \cdots \right)\\
&=& 2 \left(\frac{2^m}{2^m -1}\right) \left(\sup_{x\in \R}\abs{b(x)}\right) \left( \sup_{I\in \D} \abs{\epsilon_I}\right) \prod_{j=1}^m \La \abs{f_j} \Ra_{I_x}\\
&\leq&  \frac{2^{m+1}}{(2^m -1)} \left(\sup_{x\in \R}\abs{b(x)}\right) \left( \sup_{I\in \D} \abs{\epsilon_I}\right)\mathcal{M}(\vec{f\,})(x).
\end{eqnarray*}
The same proof works for $x <0$ with $I_x$ the smallest dyadic interval that contains both $x$ and the interval $S''.$ \\
\end{proof}
\begin{thm}\label{MTC}
Let $\vec{\alpha} \in U_m $ and $\epsilon =(\epsilon_I)_{I\in \D}$ be bounded. Suppose $b\in BMO^d$ and $\vec{w} = (w_1,\ldots,w_m) \in A_{\vec{P}}^d$ for $\vec{P}=(p_1, \ldots,p_m)$ with $ \frac{1}{p_1}+\dots+ \frac{1}{p_m} = \frac{1}{p}$ and $1 < p_1, \ldots,p_m < \infty.$ Then there exists a constant $C$ such that
\begin{equation} \label{EC}\left\Vert{[b,T_\epsilon^{\vec{\alpha}}]_i(\vec{f})}\right\Vert_{L^{p}(\nu_{\vec{w}})} \leq C \norm{b}_{BMO^d} \prod_{j=1}^m \norm{f_j}_{L^{p_j}(w_j)}.\end{equation}
\end{thm}
\begin{proof} First assume that $b$ is bounded.\\

\noindent Since the simple functions in $L^p(\nu_{\vec{w}})$ are dense in $L^p(\nu_{\vec{w}}),$ it suffices to prove \eqref{EC} for $\vec{f} = (f_1,f_2,\ldots,f_m)$ with $f_i \in L^{p_i}(w_i)$ simple. For all such $\vec{f}$, there exists, by Lemma \ref{MFDC}, a $\delta \in (0, 1/m)$ such that $\left\Vert {M}_\delta\left([b,T_\epsilon^{\vec{\alpha}}]_i(\vec{f\,})\right)\right\Vert_{L^p(w)} < \infty$.  So, for any $r>1$ and $\gamma >\delta$ we have
\begin{eqnarray*}
\left\Vert{[b,T_\epsilon^{\vec{\alpha}}]_i(\vec{f})}\right\Vert_{L^{p}(\nu_{\vec{w}})}
 &\leq& \left\Vert{M_\delta[b,T_\epsilon^{\vec{\alpha}}]_i(\vec{f})}\right\Vert_{L^{p}(\nu_{\vec{w}})}\\
&\lesssim& \left\Vert{M_\delta^\#[b,T_\epsilon^{\vec{\alpha}}]_i(\vec{f})}\right\Vert_{L^{p}(\nu_{\vec{w}})}\\
&\lesssim& \norm{b}_{BMO^d} \left(\left\Vert{\mathcal{M}_r(\vec{f})}\right\Vert_{L^{p}(\nu_{\vec{w}})} + \left\Vert{M_\gamma \left(T_\epsilon^{\vec{\alpha}}(\vec{f}\,) \right)}\right\Vert_{L^{p}(\nu_{\vec{w}})}\right),
\end{eqnarray*}
where the first inequality follows from the pointwise control, the second one is the Fefferman-Stein inequality \eqref{FSS} and the last inequality follows from Theorem \ref{MFC}.\\

\noindent 
Now we can choose $\gamma \in (\delta, 1/m)$ such that $\left\Vert{M_\gamma \left(T_\epsilon^{\vec{\alpha}}(\vec{f}\,) \right)}\right\Vert_{L^{p}(\nu_{\vec{w}})} < \infty.$ In fact, looking at the proofs of Lemmas \ref{MFD} and \ref{MFDC}, any $\gamma \in (\delta, p/p_0)$ would work. For such $\gamma,$ we have
\begin{eqnarray*}
\left\Vert{M_\gamma \left(T_\epsilon^{\vec{\alpha}}(\vec{f}\,) \right)}\right\Vert_{L^{p}(\nu_{\vec{w}})}
&\lesssim& \left\Vert{M_\gamma^\# \left(T_\epsilon^{\vec{\alpha}}(\vec{f}\,) \right)}\right\Vert_{L^{p}(\nu_{\vec{w}})}\\
&\leq& \left\Vert{\mathcal{M} (\vec{f}\,)}\right\Vert_{L^{p}(\nu_{\vec{w}})}\\
&\leq& \left\Vert{\mathcal{M}_r (\vec{f}\,)}\right\Vert_{L^{p}(\nu_{\vec{w}})}
\end{eqnarray*}
We thus have $$\left\Vert{[b,T_\epsilon^{\vec{\alpha}}]_i(\vec{f})}\right\Vert_{L^{p}(\nu_{\vec{w}})} \lesssim \norm{b}_{BMO^d}\left\Vert{\mathcal{M}_r (\vec{f}\,)}\right\Vert_{L^{p}(\nu_{\vec{w}})}$$ for all $r>1.$\\
Finally, we can choose $r>1$ such that the inequality \eqref{Mr} holds, i.e.$$ \left\Vert{\mathcal{M}_r (\vec{f}\,)}\right\Vert_{L^{p}(\nu_{\vec{w}})} \lesssim \prod_{j=1}^m \norm{f_j}_{L^{p_j}(w_j)}.$$ This completes the proof when $b$ is bounded.\\

\noindent
\noindent Now following \cite{LOPTT}, we use a limiting argument to prove the theorem for general $b \in BMO^d.$ \\
Let $\{b_j\}$ be the sequence of functions defined by
$$ b_j(x) = \begin{cases}
       j, &\text{if }b(x) > j,\\
       b(x),  & \text{if } \abs{b(x)} \leq j,\\
			-j & \text{if } b(x) < -j.\\
            \end{cases}$$
Clearly, $b_j \rightarrow b$ pointwise, and we have $\norm{b_j}_{BMO^d} \leq c \norm{b}_{BMO^d}$ for all $j$. In fact, c = 9/4 works (see \cite{LGM}, page 129).\\
For any $q \in (1,\infty)$,  
$$ T_\epsilon^{\vec{\alpha}}(f_1, \ldots, b_jf_i,\ldots,f_m) \rightarrow T_\epsilon^{\vec{\alpha}}(f_1, \ldots, bf_i,\ldots,f_m) \quad \text{ in } L^q \text{ as } j \rightarrow \infty$$ due to boundedness of $T_\epsilon^{\vec{\alpha}}: L^{mq}\times \dots \times L^{mq} \rightarrow L^q$ and the fact that bounded functions $f_1, \ldots, f_m$ with compact support are all in $L^{mq}.$ Note that since $b_j, b \in BMO^d$ and bounded function $f_i$ has compact support $b_j f_i \rightarrow bf_i$ in $L^{mq}$ as $j\rightarrow \infty.$ Then there exists a subsequence $\{b_{j_k}\}$ such that $$ T_\epsilon^{\vec{\alpha}}(f_1, \ldots, b_{j_k}f_i,\ldots,f_m)(x) \rightarrow  T_\epsilon^{\vec{\alpha}}(f_1, \ldots, bf_i,\ldots,f_m)(x) \text{\quad  for almost every } x.$$ For such $x,$ we have 
$$ [b_{j_k}, T_\epsilon^{\vec{\alpha}}]_i(\vec{f})(x) \rightarrow  [b, T_\epsilon^{\vec{\alpha}}]_i(\vec{f})(x).$$
Now, 
\begin{eqnarray*}
\left\Vert{[b,T_\epsilon^{\vec{\alpha}}]_i(\vec{f})}\right\Vert_{L^{p}(\nu_{\vec{w}})}
&=& \left(\int_\R \left\vert [b, T_\epsilon^{\vec{\alpha}}]_i(\vec{f})(x) \right\vert^p dx \right)^{1/p}\\
&\leq& \liminf_{k \rightarrow \infty}\left(\int_\R \left\vert [b_{j_k}, T_\epsilon^{\vec{\alpha}}]_i(\vec{f})(x) \right\vert^p dx\right)^{1/p}\\
&\leq& C' \liminf_{k \rightarrow \infty} \norm{b_{j_k}}_{BMO^d} \prod_{j=1}^m \norm{f_j}_{L^{p_j}(w_j)}\\
&\leq& C \norm{b}_{BMO^d} \prod_{j=1}^m \norm{f_j}_{L^{p_j}(w_j)},\\
\end{eqnarray*}
where we have used Fatou's lemma to obtain the first inequality, and the second inequality follows from the result already proved for bounded function $b$.
\end{proof}

\noindent
\textbf{Some Remarks:}
\begin{enumerate}
	\item In \cite{IJK} we have presented the unweighted theory of the multilinear commutators with some restrictions. In that paper, we required that $b \in L^q$ for some $q\in (1,\infty)$ and that $p>1.$ As we have seen, this restricted unweighted theory was sufficient to obtain the weighted theory presented in this article. Taking $w_i = 1$ for all $1\leq i \leq m$, we see that the weighted theory implies the unweighted theory for all $b \in BMO^d$ and $ 1/m < p < \infty.$
	
	\item With the results obtained in this article, it is easy to see that the end-point results obtained in \cite{LOPTT} for the commutators of  the multilinear Calder$\acute{\text{o}}$n-Zygmund operators also hold for the commutators of the multilinear Haar multipliers. 
\end{enumerate}

\end{document}